\documentclass[12pt]{article}
\usepackage{amsmath, amsfonts, amssymb, amsthm}
\usepackage[colorlinks = true, linkcolor = blue]{hyperref}
\newtheorem{theorem}{Theorem}[section]
\newtheorem{proposition}[theorem]{Proposition}
\newtheorem{definition}[theorem]{Definition}
\newtheorem{remark}[theorem]{Remark}

\newtheorem{example}[theorem]{Example}
\newtheorem{question}[theorem]{Question}

\newtheorem{corollary}[theorem]{Corollary}
\newtheorem{lemma}[theorem]{Lemma}

\def\R{\mathbb R}

\def\dim{{\rm{dim}}\,}

\def\Ad{\hbox{\rm Ad}}

\def\Aut{\hbox {\rm Aut}}

\def\g{\hbox{\sf  g}}

\def\<{\,<\!}
\def\>{\!>\,}

\def\Aut{\hbox {\rm Aut}}

\def\R{\mathbb{R}}

\def\g{\mathfrak{g}}

\def \Ad {{\rm{Ad}}}

\def \exp {{\rm exp}}

\def \dim {{\rm{dim \;}}}
\def \g {\mathfrak{g}}

\def \R {\mathbb{R}}

\begin{document}
	
	\title{On stable Cartan subgroups of Lie groups}
	
	\author{Parteek Kumar, Arunava Mandal, and Shashank Vikram Singh}

	\maketitle
	
	%
	

	\begin{abstract}
		Let $G$ be a connected real Lie group with associated Lie algebra $\mathfrak g$, and let ${\rm Aut}(G)$ be the group of (Lie) automorphisms of $G$. It is noted here that, given a super-solvable subgroup $\Gamma\subset {\rm Aut}(G)$ of semisimple automorphisms, there exists a $\Gamma$-stable Cartan subgroup, by using a result of Borel and Mostow. We characterize the $\Gamma$-stable  Cartan subgroups (with induced action) in the quotient group modulo a $\Gamma$-stable closed normal subgroup as the images of the $\Gamma$-stable Cartan subgroups in the ambient group. It is well known that a semisimple automorphism of $\mathfrak g$ always fixes a Cartan subalgebra of $\mathfrak g$. Conversely, if we take a representative from each non-conjugate class of Cartan subalgebras in a real Lie algebra, we show that there exists a non-identity automorphism that fixes these representatives. We explicitly identify such automorphisms in the case of classical simple Lie algebras. 
        As a consequence, we deduce an analogous result for semisimple Lie groups. Moreover, given a $\Gamma$-stable Cartan subgroup $H$ of $G$, and a $\Gamma$-stable closed connected normal subgroup $M$ of $G$, we prove that there exists a $\Gamma$-stable Cartan subgroup $H_M$ of $M$ such that $H\cap M\subset H_M$. 
        \end{abstract}


	
 \section{Introduction}
 Let $G$ be a real connected Lie group. In this article, we study the automorphisms of $G$ that fix certain closed subgroups of $G$. The automorphism group is itself a Lie group (not necessarily connected). The natural action of ${\rm Aut}(G)$ on the subgroups of $G$ plays an important role in the study of various topics, including geometry, dynamics, ergodic theory, and also contributes to solving polynomial equations (so called the surjectivity of the power map or word map problem) in the Lie groups.

 A seminal work by A. Borel and G. D. Mostow on semisimple automorphisms of Lie algebras facilitates the proof of various results on certain stable subalgebras of a Lie algebra under the group of automorphisms. It has a significant impact on a wide range of mathematical studies, such as the invariance of maximal tori, Cartan subgroups, and Borel subgroups, which are subgroups of interest to a large extent in the theory of algebraic groups. Additionally, it turns out to be a useful tool in determining some of the questions on the surjectivity of the power maps, as well as the density of their images, among other things. The above-mentioned theme and its related areas have attracted a lot of attention (see e.g. \cite{BhM}, \cite{BM}, \cite{BS}, \cite{Ch}, \cite{D}, \cite{KM}, \cite{Ma} \cite{St}, \cite{W}, \cite{Wu} ). The theme of fixed point sets and stable subgroups by automorphisms of a connected linear algebraic group was carried forward by numerous mathematicians, for example, J. P. Serre \cite{BS}, G. D. Mostow \cite{M}, J. H. Winger \cite{W}, R. Steinberg \cite{St}, P. Chatterjee \cite{Ch}, and many others. The majority of the results in this regard are of a purely algebraic nature.

 Let $\mathfrak g$ be the Lie algebra of a connected Lie group $G$. An automorphism $\psi: G\to G$ induces an automorphism $d\psi:\mathfrak g\to\mathfrak g$, and so it induces an injective map ${\rm Aut}(G) \to {\rm Aut}(\mathfrak g)\subset {\rm GL}(\mathfrak g)$ for a connected Lie group $G$. If $G$ is simply connected, then ${\rm Aut}(G)$ is isomorphic to ${\rm Aut}(\mathfrak g)$, but it is not true, in general. A subgroup $\Gamma$ of ${\rm Aut}(G)$ (or $\Gamma$ in ${\rm Aut}(\mathfrak g)$) is said to be super-solvable if there exists a finite sequence of subgroups $\Gamma=\Gamma_0\supset\Gamma_1\supset\cdots\supset\Gamma_r=\{e\}$
 such that each $\Gamma_i$ is normal in $\Gamma$, and each successive quotient $\Gamma_i/\Gamma_{i+1}$ is cyclic.

 Let $\Gamma$ be a subgroup of semisimple automorphisms of a Lie algebra $\mathfrak g$ over a field of characteristic zero. In \cite{BM}, A. Borel and G. D. Mostow showed that if $\mathfrak g$ and $\Gamma$ are solvable, then there exists a $\Gamma$-stable Cartan subalgebra. Furthermore, if $\Gamma$ is a super-solvable subgroup of ${\rm Aut}(\mathfrak g)$, consisting of semisimple automorphisms, then there exists a Cartan subalgebra of $\mathfrak g$ that is $\Gamma$-stable.
 In the case when $G$ is a connected solvable algebraic group over an algebraically closed field $\mathbb F$ with arbitrary characteristic, D. J. Winter showed that a finite subgroup $\Gamma\subset {\rm Aut}(G)$ with ${\rm Ord}(\Gamma)$ (cardinality of the finite group $\Gamma$) is co-prime to the characteristic of $\mathbb F$, $\Gamma$ leaves a Cartan subgroup of $G$ invariant (or stable) \cite{W}. In particular, if $\sigma$ is a semisimple algebraic automorphism of a connected algebraic group $G$, then $\sigma$ keeps stable a Cartan subgroup of $G$ (see \cite{W}). In \cite{St}, R. Steinberg has independently shown that $\sigma$ stabilizes a Borel subgroup of $G$. In \cite{Ch}, P. Chatterjee extends this result in an algebraic group setup. He proved that if the characteristic of $\mathbb F$ is zero, $G$ is an algebraic group defined over $\mathbb F$, and $\Gamma$ is a super-solvable subgroup of the semisimple $\mathbb F$-automorphisms of $G$, then there is a $\Gamma$-stable Cartan subgroup $H$ of $G$. Recall that an automorphism $\psi$ of an algebraic group $G$ defined over a field $\mathbb F$ is called semisimple if the derivative $d\psi$ is a semisimple $\mathbb F$-linear map on the Lie algebra of $G$ over $\mathbb F$ (\cite{Ch}).

 In this article, we investigate the above theme in the setup of the Lie group and study automorphism-invariant Cartan subgroups of a connected Lie group $G$. A subgroup $H$ in a Lie group $G$ is said to be a Cartan subgroup if it is a maximal nilpotent subgroup satisfying an additional property that every normal subgroup of finite index in $H$ is also of finite index in its own normalizer in $G$. We say, an automorphism $\psi$ of a connected Lie group $G$ is a semisimple automorphism if the derivative $d\psi$ is a semisimple linear map on its Lie algebra $\mathfrak{g}$. Using a result by Borel and Mostow, we extend Theorem 1.1 of \cite{Ch}.

 \begin{proposition}\label{Lie group}
   Let $G$ be a connected Lie group. Let $\Gamma\subset {\rm Aut}(G)$ be a super-solvable subgroup of semisimple automorphisms.
Then there exists a $\Gamma$-stable Cartan subgroup of $G$.  
\end{proposition}

 For a connected Lie group $G$, there is a (unique) largest connected solvable normal subgroup $R$ of $G$, called the radical of $G$. Then any connected Lie group can be written as $G=SR$, where $S$ is a Levi subgroup, which is a maximal connected semisimple subgroup of $G$ unless it is
trivial. Recently, in \cite{MS21} (and \cite{MS23}), it is proved that, for a closed connected normal subgroup $M$ of $G$, if $H$ is a Cartan subgroup of $G$, then $HM/M$ is a Cartan subgroup of $G/M$. Conversely, if $Q$ is a Cartan subgroup of $G/M$, then there exists a Cartan subgroup $H$ of $G$ such that $HM/M=Q$. 
 Note that radical $R$ is stable under any automorphism of $G$, and an automorphism $\psi:G\to G$ induces an automorphism $\hat\psi: G/R\to G/R$ by $\hat\psi(gR):=\psi(g)R$. For a subgroup $\Gamma\subset {\rm Aut}(G)$, we denote $\hat\Gamma=\{\hat\psi\;|\;\psi\in\Gamma\}$, the induced automorphism on $G/R$.
Let $\Gamma\subset {\rm Aut}(G)$ be a super-solvable subgroup of semisimple automorphisms of $G$. In Proposition \ref{quotient}, we prove that if $Q$ is $\hat\Gamma$-stable Cartan subgroup of $G/R$, there exists a $\Gamma$-stable Cartan subgroup $H$ of $G$ such that $HR/R=Q.$ The result ensures that to get an automorphism stable Cartan subgroup of a (real) connected Lie group $G$, it is enough to look for (the induced) automorphism stable Cartan subgroups of the semisimple group $G/R$. Furthermore, we extend this result for any closed $\Gamma$-stable normal subgroup $M$ of $G$ (instead of the radical $R$ of $G$), by using Proposition \ref{quotient}. 

\begin{theorem}\label{quotient modulo M}
    Let $G$ be a connected Lie group, and let $\Gamma\subset {\rm Aut}(G)$ be a super-solvable subgroup of semisimple automorphisms. Let $M$ be a closed normal subgroup of $G$ stable under $\Gamma$.
     Let $\hat\Gamma$ be the group of automorphisms on $G/M$ induced from $\Gamma$. Let $Q$ be a $\hat \Gamma$-stable Cartan subgroup of $G/M$, then there exists a $\Gamma$-stable Cartan subgroup $H$ of $G$ such that $HM/M=Q.$
\end{theorem}

Let $K_0R/R$ be a maximal compact connected normal subgroup of $G/R$. Observe that the subgroup $K_0R$ of $G$, which is amenable, is characteristic in $G$, and hence it is automorphism invariant. Therefore, by a direct application of Theorem \ref{quotient modulo M}, it suffices to examine the automorphism-invariant Cartan subgroups of the non-compact semisimple Lie group $G$. To begin, we focus on semisimple Lie algebras and derive the corresponding results for semisimple Lie groups.

The group $G$ acts on its Lie algebra $\mathfrak g$ via the adjoint action. Suppose $g\in G$ is such that ${\rm Ad}(g):\mathfrak g\to\mathfrak g$ is semisimple. Then, in this case, it is well-known that ${\rm Ad}(g)$ fixes a Cartan subalgebra of $\mathfrak g.$ Here, we explore the converse of this statement. Note that there are finitely many conjugacy classes of Cartan subalgebras of a real semisimple Lie algebra $\mathfrak g.$ Let $\mathfrak h_1,\ldots,\mathfrak h_n$ be the representatives of each distinct conjugacy class of Cartan subalgebras. 
We ask the following question.

\begin{question}
    Let $\mathfrak g$ be a real semisimple Lie algebra. Does there exist a non-identity automorphism $\sigma\in\Aut(\mathfrak g)$ such that $\sigma(\mathfrak h_i)=\mathfrak h_i$ for all $i=1,\ldots,n$ ?
\end{question}

We give an affirmative answer to the above question that follows from Theorem \ref{s.s liealgebra}. To state Theorem \ref{s.s liealgebra}, we require some preliminaries. Let $\g$ be a real semisimple Lie algebra, and $\sigma$ be the Cartan involution of $\g$. Then 
    $\g=\mathfrak{k}+ \mathfrak{p},$ where $\mathfrak{k}=\{X\in \g: \sigma X=X\}$ and $\mathfrak{p}=\{X\in \g: \sigma X=-X\}$ satisfying $[\mathfrak{k},\mathfrak{k}]\subset\mathfrak{k}$, $[\mathfrak{k},\mathfrak{p}]\subset \mathfrak{p}$ and $[\mathfrak{p},\mathfrak{p}] \subset \mathfrak{k}$. Any two Cartan decompositions are conjugate to each other under the action of $\mathrm{Ad}(G)$, where $\mathrm{Ad}(G)$ denotes the adjoint group of $\g$. Let $K$ be the analytic subgroup of $ G$ with Lie algebra $\mathfrak{k}$ and $\mathfrak{m}$ be a maximal abelian subalgebra in $\mathfrak{p}$.  A triple $(\mathfrak{k},\mathfrak{p},\mathfrak{m})$ is said to be a standard triple if $\g=\mathfrak{k}+\mathfrak{p}$ is a Cartan decomposition and $\mathfrak{m}$ is a maximal abelian subalgebra in $\mathfrak{p}$. With respect to the standard triple $(\mathfrak{k},\mathfrak{p},\mathfrak{m})$, where $\mathfrak{m}^-=\mathfrak{m}$ (and $\mathfrak{m}^-=\{X\in \mathfrak{m}$: all eigenvalues of $ad X$ are purely imaginary$\}$), there is a one to one correspondence between conjugate classes of Cartan subalgebras under $K$ of a real semisimple Lie algebra $\mathfrak g$ and the conjugate classes of admissible root systems contained in $\mathbf{R}(\mathfrak m)$ (Root system associated to the Cartan subalgebra $\mathfrak{m}$) under $W(\bf R)$ (see Theorem 6 of \cite{Su}), the detailed notations and terminologies for Theorem \ref{s.s liealgebra} are given in \S 2 (or see \cite{Su}).

\begin{theorem}\label{s.s liealgebra}
    Let $(\mathfrak{k},\mathfrak{p},\mathfrak{m})$ be a standard triple of a real simple Lie algebra $\mathfrak{g}$, where $\mathfrak{m}^-=\mathfrak{m}$. 
Let $G$ be the connected Lie group associated to $\mathfrak g$, and
let $K$ be the analytic subgroup of $ G$ with Lie algebra $\mathfrak{k}$. 
Then the following hold:
\begin{enumerate}
\item[{$(1)$}] If $s\in W(\bf R)$ fixes all admissible root systems, then there exists $k\in K$ and a representative $\mathfrak h_1,\cdots,\mathfrak h_m$ of each distinct conjugacy classes of the Cartan subalgebras such that $k\cdot \mathfrak h_i=\mathfrak h_i$ for all $i=1,\ldots,m$.

\item[{$(2)$}] If $g\in G$ fixes a representative of each distinct conjugacy classes of the Cartan subalgebras $\mathfrak h'_i$,  then there exists an element $s \in W(\mathbf{R}$) that fixes a representative of all admissible root systems $F_1, F_2, \ldots, F_m$ for all $i$.
\end{enumerate}
\end{theorem}

Note that there exists a unique longest Weyl group element that sends every positive root to its negative one. It is straightforward to see that this element fixes all the admissible root systems. Then part (1) of Theorem \ref{s.s liealgebra} ensures that there exists an element $k \in K$ that fixes all the Cartan subalgebras $\mathfrak{h_i}$. 
In addition, we explicitly identify such $k \in K$ for the classical simple Lie algebras $A_n$, $B_n$, $C_n$, and $D_n$ types (see \S 4). In some cases, for e.g. $\mathfrak g=\mathfrak {sl}(n,\mathbb R)$, $\mathfrak o(2n,\mathbb R)$ with $n\leq 5$, we show that there are many ($\geq 2$) non-identity automorphisms of the Lie algebra $\mathfrak{g}$ that fix all Cartan subalgebras $\mathfrak h_1,\cdots,\mathfrak h_n$, whereas in some cases (for e.g., $\mathfrak {sp}(n,\mathbb R)$, $\mathfrak o(2n,\mathbb R)$ with $n\geq 6$) we prove that there exists a unique (non-identity) element in $K$ that does the job. Also for $\mathfrak o(2n+1,\mathbb R)$ with $n$ odd, we show that there are exactly three (non-identity) elements in $K$ that fix all $\mathfrak h_i$ (see \S 4 for details). Furthermore, we deduce the following result.

\begin{corollary}
    Let $G$ be a connected  semisimple Lie group with Lie algebra $\mathfrak{g}$ and let $H_1,\ldots, H_m$ be Cartan subgroups of $G$ such that $Lie(H_i)=\mathfrak{h_i}$, where $\mathfrak{h_i}$ are as described in Theorem \ref{s.s liealgebra}. Then there exists $k\in K$ such that $k$ fixes each $H_i$, that is, $kH_ik^{-1}=H_i$ for all $i$.
\end{corollary}

Let $G$ be a connected Lie group with radical $R$ and nilradical $N$ (the largest connected nilpotent normal subgroup of $G$). Let $H$ be a Cartan subgroup of $G$. Then $H\cap R$ is not necessarily a Cartan subgroup of $R$. However, $H\cap R$ is contained in a Cartan subgroup $H_R$ of $R$. The same statement holds when $M$ is any connected, normal subgroup of $G$ not just the radical $R$ (\cite{MS21}). We now extend this result as follows:

 \begin{theorem}\label{intersection with sub}
 Let $G$ be a connected Lie group and let $\Gamma$ be a subgroup of {\rm Aut}(G).
 Let $M$ be a $\Gamma$-stable closed connected normal subgroup, and let $H$ be a $\Gamma$-stable Cartan subgroup of $G$. Then there exists a $\Gamma$-stable Cartan subgroup $H_M$ of $M$ such that $H\cap M\subset H_M$
 \end{theorem}

\begin{remark}
   It is known that if $H$ is a Cartan subgroup in $G$, then there is a semisimple Lie subgroup $S$ and the radical $R$ such that $G=SR$, and $H=(H\cap S)(H\cap R)$. Suppose that $M$ is a $\Gamma$-stable closed connected normal subgroup containing $R$. Take $J=(S\cap M)^0$, and hence $M=JR$. Results from \cite{MS21} show that every Cartan subgroup of $M$ is of the form $H_JH_{Z_R(H_J)}$, where $H_J$ is a Cartan subgroup of $J$ and $H_{Z_R(H_J)}$ is a Cartan subgroup of $Z_R(H_J)$. However, this description is inadequate for our needs because $H_J$ is not necessarily invariant under $\Gamma$. To prove Theorem \ref{intersection with sub}, we therefore need a Cartan subgroup of $M$ that is $\Gamma$-stable and also contains $H\cap M$. We achieve this by considering the subgroup $Z_M(H')H'$, where $H'$ is a $\Gamma$-stable Cartan subgroup of $H_J^0R.$ 
   \end{remark}
 
The paper is organized as follows. In \S 2, we recall some definitions and useful facts. In \S 3, we prove Proposition \ref{Lie group} and Proposition \ref{quotient}. Theorem \ref{s.s liealgebra} is taken up in \S 4. We prove Theorem \ref{intersection with sub} in \S 5.

\section{Preliminary}
Let $\g$ be a real semisimple Lie algebra and let $\sigma$ be the Cartan involution of $\g$ such that
 \begin{center}
    $\g=\mathfrak{k}+ \mathfrak{p},$
\end{center} satisfying  $[\mathfrak{k},\mathfrak{k}]\subset\mathfrak{k}$, $[\mathfrak{k},\mathfrak{p}]\subset \mathfrak{p}$ and $[\mathfrak{p},\mathfrak{p}] \subset \mathfrak{k}$. Any two Cartan decompositions are conjugate to each other under the adjoint action of $ G$. Let $\mathfrak{m}$ be a maximal abelian subalgebra in $\mathfrak{p}$. Any two maximal subalgebras $\mathfrak{m_1}, \mathfrak{m_2}$ in $\mathfrak{p}$ are conjugate under the action of $K$. We recall that a Cartan subalgebra $\mathfrak{h}$ of $\g$ is called standard with respect to a standard triple $(\mathfrak{k},\mathfrak{p},\mathfrak{m})$ if the toridal part $\mathfrak{h}^+$ of $\mathfrak{h}$ ($\mathfrak{h}^+=\{X\in \mathfrak{h}$: all eigen values of $ad X$ are real$\}$) is conatained in $\mathfrak{k}$ and the vector part $\mathfrak{h}^-$ of $\mathfrak{h}$ ($\mathfrak{h}^-=\{X\in \mathfrak{h}$: all eigen values of $ad X$ are purely imaginary$\}$) is contained in $\mathfrak{m}$.

It follows from [\cite{Su}, Theorem 2] that it suffices to characterize the standard Cartan subalgebra with respect to a fixed triple $(\mathfrak{k},\mathfrak{p},\mathfrak{m})$. Additionally, two standard Cartan subalgebras are conjugate under the adjoint action of the group $G$ if and only if they are conjugate under the adjoint action of $K$ (see Corollary 1 of \cite{Su}).  Therefore, we fix the following setup for further analysis. Let $\g$ be a real semisimple Lie algebra and $\sigma(X)=-X^{t}$ (transpose of $-X$) for all $X\in \g$. Then, fix a triple $(\mathfrak{k},\mathfrak{p},\mathfrak{m})$, where $\mathfrak{k}$ contains skew-symmetric matrices of $\g$,  $\mathfrak{p}$ contains symmetric matrices of $\g$ and $\mathfrak{m}$ contains diagonal matrices of $\mathfrak{p}$. With respect to the fixed triple, Sugiura in \cite{Su} identifies the non-conjugate classes of Cartan subalgebras of $\mathfrak{g}$ (under the action of $K$). Consider a standard Cartan subalgebra $\mathfrak{h_0}$  of $\g$ such that $\mathfrak{h}^-_0=\mathfrak{m}$. Let $\mathbf{R}(\mathfrak{h_0})$ be the root system associated to $\mathfrak{h_0}$ and $W(\mathbf{R})$ be the Weyl group generated by reflections through the simple roots.

We now recall  [\cite{Su}, Theorem 5], which provides a recipe to cook Cartan subalgebras from a subset of a root system.

\begin{theorem}\label{constuctionofCartan}
    
 Let $\g$ be a real semisimple Lie algebra, and $(\mathfrak{k},\mathfrak{p},\mathfrak{m})$ be a standard triple of $\g$, and let $\mathfrak{h}_0$ be a Cartan sublagebra of $\g$ such that $\mathfrak{h}^-_0=\mathfrak{m}$. For any subspace $\mathfrak{n}$ of $\mathfrak{m}$, we denote by $\mathfrak{l}$ the subspace of $\mathfrak{m}$ defined by 
\begin{center}
    $\mathfrak{l}=\mathfrak{n}^{\perp}\cap \mathfrak{m}= \{X\in \mathfrak{m}: B(X,\mathfrak{n})=0\}$,
\end{center}
 where $B$ is the killing form. Then a subspace $\mathfrak{n}$ of $\mathfrak{m}$ becomes the vector part $\mathfrak{h}^-$ of a standard Cartan subalgebra $\mathfrak{h}$ if and only if there exists $l$ root ($l=\dim \mathfrak{l}$) $\alpha_1,\alpha_2, \ldots, \alpha_l$ such that 
 \begin{itemize}
     \item[{$(1)$}] $\alpha_i\pm \alpha_j \notin \mathbf{R}(\mathfrak{h_0})$, $1\leq i,j\leq l$ and $\alpha_i\pm \alpha_j\ne 0 $ if $i\ne j$.
     \item[{$(2)$}] $\mathfrak{l}=\sum^{l}_{i=1}\mathbb{R} H_{\alpha_i}$,
 \end{itemize}
 where $H_{\alpha_i}$ is unique element of $\mathfrak{h}_0$ associated to a root $\alpha_i$ such that $\alpha_i(H)=B(H_{\alpha_i},H)$ for all $H \in \mathfrak{h}_0$.
\end{theorem}
\begin{definition}
    A set of positive roots $F=\{\alpha_1,\alpha_2,\ldots, \alpha_l\}$ that satisfy the condition (1) of the above theorem is called an admissible root system.
\end{definition}

\begin{remark}{\label{preserverAdmissibleRootSystem}}
   We say that $\sigma \in W(\mathbf{R}) $ fixes an admissible root system $F$ if the subspace $\mathfrak{l}$ (see condition (2) of the above theorem) of $\mathfrak{g}$ remains the same for both $F$ and $\sigma(F)$.
\end{remark}

\noindent\textbf{ Notations:}
Let $G$ be a Lie group with the identity $e$ and let $M$ be a subgroup of $G$. Let $M^0$ denote the connected component 
of the identity $e$ in $M$,  $[M, M]$ denote the commutator subgroup of $M$, and let $Z(M)$ denote the center of $M$. The $N_G(M)$ 
and \ $Z_G(M)$ denote the normalizer and  centralizer of $M$ in $G$, respectively. If $M$ is a characteristic subgroup, then $N_G(M)$ and  $Z_G(M)$ are also characteristic subgroups.

  \section{Automorphism invariant Cartan subgroup of Lie groups}
Let $G$ be a Lie group with associated Lie algebra $\mathfrak g$. A subalgebra $\mathfrak h$ of $\mathfrak g$ is said to be a Cartan subalgebra if it is a maximal nilpotent Lie subalgebra and its normalizer in $\mathfrak g$ is itself, i.e., $N_\mathcal G(\mathfrak h)=\mathfrak h.$  Chevalley showed that there is a one-to-one correspondence between Cartan subgroups of $G$ and Cartan subalgebras of $\mathfrak g$. That is, the Lie algebra of any Cartan subgroup of $G$ is
a Cartan subalgebra of $\mathfrak g$, and conversely, given a Cartan subalgebra
$\mathfrak h$ of $\mathfrak g$, there exists a unique Cartan subgroup $H$ of $G$ whose Lie algebra is $\mathfrak h$.
We recall the following result by Borel and Mostow (\cite{BM}).

\begin{theorem}\label{BM}
Let $\mathfrak g$ be a Lie algebra, and let $\Gamma\subset {\rm Aut}(\mathfrak g)$ be a super-solvable subgroup of semisimple automorphisms. Then there exists a Cartan subalgebra $\mathfrak h$ of $\mathfrak g$ such that $\Phi(\mathfrak h)=\mathfrak h$, for all $\Phi\in\Gamma.$
\end{theorem}

Now, we prove Proposition \ref{Lie group} by using Theorem \ref{BM}.

\begin{proof} [Proof of Proposition \ref{Lie group}:]
Let $\hat \Gamma=\{d\psi|\psi\in\Gamma\}.$ Then $\hat\Gamma$ is a supper solvable subgroup of semisimple automorphisms in $\rm Aut(\mathfrak g)$ . Then there exists a Cartan subalgebra $\mathfrak h$ of $\mathfrak g$ such that $d\psi(\mathfrak h)=\mathfrak h$ for all $d\psi\in\hat\Gamma$ by Theorem \ref{BM}.
Let $H$ be the Cartan subgroup of $G$ such that $L(H)=\mathfrak h.$  Note that $\psi(H)$ is a Cartan subgroup of $G$, and $L(\psi(H))=d\psi(\mathfrak h)=\mathfrak h=L(H).$ Therefore, both $\psi(H)$ and $H$ are Cartan subgroups of $G$ with the same Lie algebra $\mathfrak h$. Since there is a one-to-one correspondence between Cartan subgroups of $G$ and Cartan subalgebras of $\mathfrak g$, we have $\psi(H)=H$. This proves that $H$ is $\Gamma$-stable.
\end{proof}

To prove Theorem \ref{quotient modulo M}, we begin with the following lemma.

\begin{lemma}\label{connected component}
Let $G$ be a connected Lie group, and let $H$ be a Cartan subgroup of $G$. Then the following statements hold:
\begin{enumerate}
\item[{$(1)$}] Let $M$ be a closed connected subgroup of $G$. If $H^0\subset M$, then $H\cap M$ is a Cartan subgroup of $M$. 
 
\item[{$(2)$}] Let $R$ be the radical, and $N$ be the nilradical of $G$. Then $H^0R=H^0N$ is closed and $H^0$ is a Cartan subgroup of $H^0R=H^0N$.

\end{enumerate}
\end{lemma}

\begin{proof}
$(1):$ Let $M$ be a closed connected subgroup of $G$ such that $H^0\subset M$. 
Let $\mathfrak h$ be a subalgebra for 
$H^0$ in the Lie algebra $L(M)\subset L(G)$. Then $\mathfrak h=L(H^0)$ is a Cartan subalgebra of $L(M)$. Now by Lemma 3.9\,(3) of \cite{MS21},
there exists a Cartan subgroup $H'$ of $M$ such that $H'=\{g \in M \mid \Ad(g)X = X \mbox{ for all } X \in \mathfrak h \}\exp(\mathfrak{h})$, 
where $\Ad(g)$ is the Lie algebra automorphism corresponding to the inner automorphism of $G$ by $g$ and $\exp:\mathfrak{m}\to M$ is the 
exponential map. Since $H^0$ is a connected nilpotent Lie subgroup, $\exp(\mathfrak{h})=H^0$, and 
$\{g \in M \mid \Ad(g)X = X \mbox{ for all } X \in \mathfrak{h}\}=Z_M(H^0)$. Therefore, $H'=Z_M(H^0)H^0$. We already know from Lemma 3.9\,(3) of 
\cite{MS21}, that $H=Z_{G^0}(H^0)H^0$ is a Cartan subgroup of $G$. Therefore, $H'=H\cap M$. 

$(2):$ It is known that $H\cap R$ is connected and $R=(H\cap R)N$. Since $H=H_S(H\cap R)$, we have $H^0R=H^0N$. 
Let $M=\overline{H^0R}$. As $H^0R/R$ is a connected component of the Cartan subgroup $HR/R$ of $G/R$, we have $H^0R/R$ is closed in $G/R$, and hence $H^0R$ is closed in $G$ (since $\pi:G\to G/R$ is continuous). Note that $H^0\subset M$, and hence
by part (1), $H\cap M$ is a Cartan subgroup of $M$.
Since $M$ is solvable, its Cartan subgroups are connected, and hence $H\cap M$ is connected. clearly, $H\cap M=H^0$ (since $M=H^0R$), and therefore, $H^0$ is a Cartan subgroup of $H^0R.$ 
\end{proof}

We now provide another proof of Lemma \ref{connected component} (2) by using Remark 3.10 of \cite{MS21}. It is observed in Remark 3.10 of \cite{MS21} that for a Cartan subgroup $H$ of $G$, $H\cap R=N_R(H^0)$. Note that $H^0R$ is a connected solvable Lie group. To show $H^0$ is a Cartan subgroup of $H^0R$, it is enough to show that $N_{H^0R}(H^0)=H^0$. Observe that, $N_{H^0R}(H^0)=H\cap R\subset H^0\subset N_{H^0R}(H^0)$, and hence $N_{H^0R}(H^0)=H^0$, as required.

\begin{lemma}\label{L}
Let $H$ be a Cartan subgroup of $G$ such that $H^0$ is $\Gamma$-stable. Then $H$ is $\Gamma$-stable.
\end{lemma}
\begin{proof}
 Since $H^0$ is $\Gamma$-invariant, we have $\psi(H^0)=H^0$ for all $\psi\in\Gamma$. By Lemma 3.9 of \cite{MS21}, $H=H^0Z_G(H^0)$, and hence it is immediate that $H$ is $\Gamma$-stable. 
 \end{proof}

\begin{proposition}\label{quotient}
   Let $G$ be a connected Lie group with the radical $R$. Let $\Gamma\subset {\rm Aut}(G)$ be a super-solvable subgroup of semisimple automorphisms.
 Let $\hat\Gamma$ be the group of automorphisms on $G/R$ induced from $\Gamma$. Let $Q$ be a $\hat\Gamma$-stable Cartan subgroup of $G/R$, then there exists a $\Gamma$-stable Cartan subgroup $H$ of $G$ such that $HR/R=Q.$ 
\end{proposition}

\begin{proof}
 Let $R$ be the radical of $G$ and let $\Gamma\subset{\rm Aut}(G)$. Since $R$ is a characteristic subgroup of $G$, it is stable under every automorphism in $\Gamma$, i.e., $R$ is $\Gamma$-stable. Consequently, the action of $\Gamma$ on $G$ descends to an induced action $\hat\Gamma \subset \Aut (G/R)$. 
It is given that $Q$ is a $\hat \Gamma$-stable Cartan subgroup of $G/R$. That is, for every $\hat\alpha\in\hat\Gamma$, we have $\hat\alpha(Q)=Q$.  This implies that $d\hat\alpha(L(Q))=L(Q)$, and therefore $d\hat\alpha(L(Q^0))=L(Q^0)$, where $Q^0$ denotes identity component of $Q$. 
 Since $\hat\alpha(Q)=Q$, it follows that $\hat\alpha(Q^0)\subset Q^0$. Further more, for $h\in Q^0$,  there exists $Y\in L(Q^0)$ such that $h=\exp(Y).$ Suppose $X\in L(Q^0)$ is such that $d\hat\alpha(X)=Y,$ and set  $g=\exp(X).$ Then clearly, $\hat\alpha(g)=h$, and it shows that $\hat\alpha(Q^0)= Q^0$

By Theorem 1.6 of \cite{MS21}, there exists a Cartan subgroup  $H$ of $G$ such that $Q=HR/R$. Since $\hat\alpha(Q)=Q$, and $\alpha(R)=R$ it follows that $\alpha(HR)=HR.$ Moreover, since $\hat\alpha(Q^0)=Q^0$ and $H^0R/R=Q^0$, we conclude that $\alpha(H^0R)=H^0R.$ Thus, both $HR$ and its connected component $(HR)^0=H^0R$ are $\Gamma$-stable. Since $H^0R$ is a connected solvable Lie group, there exists a Cartan subgroup $H_1$ of $H^0R$ which is $\Gamma$-stable by Proposition \ref{Lie group}. Note that $H^0$ is a Cartan subgroup of $H^0R$ by Lemma \ref{connected component}. Since all Cartan subgroups of a connected solvable Lie group are conjugate, there exists $g\in N\subset R$, the nilradical of $H^0R$, such that $gH^0g^{-1}=H_1$. 
Note that $gHg^{-1}$ is a Cartan subgroup of $G$ and $gHg^{-1}$ is $\Gamma$-stable by Lemma \ref{L}. For $x\in H$, we have $$gxg^{-1}R=gxR=xx^{-1}gxR=xR,$$ and hence $gHg^{-1}R/R\subset HR/R=Q$.
As both $gHg^{-1}R/R$ and $HR/R=Q$ are Cartan subgroups, we have $gHg^{-1}R/R=Q$. This proves the result.
\end{proof}

\begin{proof} [Proof of Theorem \ref{quotient modulo M}:]
    We may first assume that $M$ is connected. Indeed, since $\Gamma$ induces an automorphism $\hat\Gamma$ on $G/Z(G),$ if $Q \subset G/Z(G)$ is a $\hat\Gamma$-stable Cartan subgroup,  there exists a $\Gamma$-stable Cartan subgroup $H \subset G$  and $HZ(G)/Z(G)=Q$. Note that $M/M^0$ is the central subgroup in $G/M^0$. Therefore, to prove the desired result, it suffices to assume $M$ is connected. 

 Suppose $M$ is semisimple. We now show that the result holds under this assumption. Since the center $Z(M)$ is characteristic in $M$, it is normal in $G$. Moreover, because $M$ is semisimple, we have $Z(M)$ is discrete in $G$ and hence $Z(M)$ is central in $G$, that is, $Z(M)\subset Z(G)$.

 Since $M$ is a $\Gamma$-stable semisimple subgroup of $G$, its center $Z(M)$ is a discrete and $\Gamma$-stable subgroup, and hence lies in $Z(G)$. Therefore, we can replace $M$ by $M/Z(M)$ and $G$ by $G/Z(M)$, and hence assume that $M$ has trivial center. 
 Since $M$ is a normal semisimple center-free subgroup of $G$, we may write $G=M\times Z_G(M)$ by Lemma 3.9 of \cite{I}. As $M$ is $\Gamma$-stable there exists a $\Gamma$-stable Cartan subgroup $H_1$ of $M$ by Proposition \ref{Lie group}. Also,  $Z_G(M)$ is $\Gamma$-stable. Since  $Q$ is a $\hat\Gamma$-stable subgroup of $G/M$ (isomorphic to $Z_G(M)$), we may assume that $H_2:=Q$ is a $\Gamma$-stable Cartan subgroup of $Z_G(M)$. Clearly, $H:=H_1\times H_2$ is $\Gamma$-stable Cartan subgroup of $G$. This proves the result under this assumption.

 Suppose $M$ is a solvable subgroup contained in the radical $R$ of $G$. Then the result follows from Proposition \ref{quotient}.

 Suppose that $M$ is neither solvable nor semisimple. Let $M_R$ denote the radical of $M$. Then we have a surjection of maps
$$G\to G/M_R\to (G/M_R)/(M/M_R).$$

Let $\Tilde{\Gamma}$ be the induced action of $\Gamma$ on $G/M_R$. Note that the induced of $\Gamma$ on $G/M$ is same as the induced action of $\Tilde{\Gamma}$ on $G/M$.
 Since $M/M_R$ is a  $\Tilde{\Gamma}$-stable semi-simple subgroup of $G/M_R$, for a given $\hat{\Gamma}$-stable Cartan subgroup $Q$ of $G/M$, there exists a ${\Tilde\Gamma}$-stable Cartan subgroup $H'$ of $G/M_R$. As $M_R$ is solvable, in view of Proposition \ref{quotient}, there exists a $\Gamma$-stable Cartan subgroup $H$ of $G$. This proves the result.
\end{proof}

\section{On classical Simple Lie algebra}
The primary objective of this section is to prove Theorem \ref{s.s liealgebra}. To facilitate this, we recall that $\mathfrak{m}^- = \mathfrak{m}$, indicating that $\mathfrak{m}$ is a Cartan subalgebra of the Lie algebra $\mathfrak{g}$.
 
Let $\mathbf{R}(\mathfrak{m})$ denote the root system associated to the Cartan subalgebra $\mathfrak{m}$, and let $W(\mathbf{R})$ represent the Weyl group generated by reflections through the simple roots in $\mathbf{R}(\mathfrak{m})$.  According to Theorem 6.57 of \cite{K}, the Weyl group $W(\mathbf{R})$ is isomorphic to the quotient group ${N_K(\mathfrak{m})}/{Z_K(\mathfrak{m})}$, where $N_K(\mathfrak{m})$ and $Z_K(\mathfrak{m})$ denote the normalizer and centralizer of $\mathfrak{m}$ in $K$, respectively. This result implies that for each element $s \in W(\mathbf{R})$, there exists an element $k \in {N_K(\mathfrak{m})}/{Z_K(\mathfrak{m})}$ such that the adjoint action of $K$ maps root spaces accordingly: $\Ad (k)(E_{\alpha}) = E_{s\alpha}$, for every root $\alpha \in$ $\mathbf{R}(\mathfrak{m})$, where $E_{\alpha}$ denotes the root space corresponding to $\alpha$. 

Let $F_1, F_2, \ldots, F_m$ be the admissible root systems and $\mathfrak{h_1},\mathfrak{h_2},\ldots, \mathfrak{h_m}$ be the corresponding Cartan subalgebras (see Theorem 6 of \cite{Su}). Suppose $s \in W(\mathbf{R})$ fixes $F_i$ for all $i$. Then, it follows from Theorem \ref{constuctionofCartan} of \cite{Su} that the vector part of the Cartan subalgebra $\mathfrak{h_i}$, for all $i$, is invariant under $k \in {N_K(\mathfrak{m})}/{Z_K(\mathfrak{m})}$. However, it is not immediate that the toroidal parts of $\mathfrak{h_i}$ are also preserved by the same $k$. Therefore, the main objective of the proof of Theorem \ref{s.s liealgebra}(1) is to construct such Cartan subalgebras associated to $F_i$'s whose toroidal parts are also invariant under $k$.

\begin{proof} [Proof of Theorem \ref{s.s liealgebra}:]

$(1):$ Let $\mathfrak{m}$ be a Cartan subalgebra of $\mathfrak{g}$, and suppose $\dim \mathfrak{m} = n$. Set an admissible root system $F_1 = \{\alpha_1, \alpha_2, \dots, \alpha_r\}$, where $\alpha_i \in \mathbf{R}(\mathfrak{m})$. Now we determine a Cartan subalgebra associated to the admissible root system $F_1$. For this purpose, we define the subspace $\mathfrak{l}=\sum^{r}_{i=1}\mathbb{R} H_{\alpha_i}$. By construction, we have $\dim \mathfrak{l} = r$ (see Theorem \ref{constuctionofCartan} for notations). Consider the orthogonal complement of $\mathfrak{l}$ in $\mathfrak{m}$ with respect to the Killing form $B$. We denote this subspace by $$\mathfrak{n}=\mathfrak{l}^{\perp}\cap \mathfrak{m}= \{X\in \mathfrak{m}: B(X, H_{\alpha_i})=0 \ \text{for all} \ i\}.$$ Then the dimension of $\mathfrak{n}$ is $n - r$.

The dimension of the root space $E_{\alpha_i}$ for all $i$ is one,  we therefore assume $E_{\alpha_i}=<x_{\alpha_i}>$ for all $i$. Now, consider the subspace generated by the elements $U_{\alpha_i}=x_{\alpha_i}+x_{-\alpha_i}$ for all $1 \leq i \leq r$. The commutation relations between these elements are given by $$[U_{\alpha_i}, U_{\alpha_j}]= c_{\alpha_i,\alpha_j}U_{\alpha_i+\alpha_j}+c_{\alpha_i,-\alpha_j}U_{\alpha_i-\alpha_j},$$ where $i \ne j$,  $c_{\alpha_i,\alpha_j}$ and $c_{\alpha_i,-\alpha_j} $ are some constants. Since $F_1$ is an admissible root system, the sum ${\alpha_i+\alpha_j}$ and the difference ${\alpha_i-\alpha_j}$ are not a root; therefore we get that $$U_{\alpha_i+\alpha_j}=0 \text{\ and \ } U_{\alpha_i-\alpha_j}=0.$$ Moreover, for any $Y \in \mathfrak{n}$, the adjoint action satisfies $$[Y,x_{\alpha_i}+x_{-\alpha_i}]= c_{\alpha_i}B(Y,H_{\alpha_i})x_{\alpha_i}+c_{-\alpha_i}B(Y,H_{-\alpha_i})x_{-\alpha_i}.$$ Since $B(Y,H_{\alpha_i})=0$  (and therefore $B(Y,H_{-\alpha_i})=0$), it follows that the space generated by $U_{\alpha_i}$ for $1 \leq i \leq r$ is contained in the centralizer $Z_{\mathfrak{k}}(\mathfrak{n})$ and is an abelian subalgebra of dimension $r$. We denote it by $\mathfrak{e}$.

Let $\mathfrak{t}$ be maximal abelian subgroup of $Z_{\mathfrak{k}}(\mathfrak{n})$ that contains $\mathfrak{e}$. Then, by the construction $\mathfrak{t}+ \mathfrak{n}$ forms a Cartan subalgebra of $\mathfrak{g}$ (see \cite{K}, Proposition 6.47). Since 
$n=\dim(\mathfrak{t})+ n-r $ and $\mathfrak{e}$ is a $r$-dimensional abelian subalgebra contained in $\mathfrak{t}$, it follows that $\mathfrak{t}=\mathfrak{e}$. Therefore, $\mathfrak{h_1}=\mathfrak{e}+\mathfrak{n}$ is a Cartan subalgebra associated to the admissible root system $F_1$. 

By Theorem 6 of \cite{Su}, the subalgebra $\mathfrak{h_1}$ is the only Cartan subalgebra associated to the fixed admissible root system $F_1$.
Similarly, construct the  Cartan subalgebras $\mathfrak{h_2},\mathfrak{h_3}, \ldots, \mathfrak{h_m}$ associated to the admissible root systems $F_2,F_3,\ldots, F_m$, respectively. Suppose $s\in W(\mathbf{R})$  is an element that fixes all the admissible root systems. By the identification of $W(\mathbf{R})$ with ${N_K(\mathfrak{m})}/{Z_K(\mathfrak{m})}$, there exists $k \in {N_K(\mathfrak{m})}/{Z_K(\mathfrak{m})} $. Here $k=k.Z_K(\mathfrak{m})$, in particular choose identity $e \in Z_K(\mathfrak{m})$ and write $k.e=k$. Then, it is clear by the identification of $s$ with $k.e=k$ and the construction of $\mathfrak{h_i}$ that $k.e=k$ fixes all the Cartan subalgebras $\mathfrak{h_i}$ for all $i$. This proves $(1)$.

 $(2):$ Given that $Ad(g)(\mathfrak{h_i'})=\mathfrak{h_i'}$ for all $i$. Then $g$ can uniquely decompose as a product $g=kp$ such that $p$ fixes $h_i'^-$ point-wise for all $i$ (see \cite{Su}, Theorem 3).  We first observe that $(g\mathfrak{h_i'}g^{-1})^{-}=g\mathfrak{h_i}^{-}g^{-1}=k\mathfrak{h_i}^{-}k^{-1}.$
  Indeed, $\mathfrak{h_i'}=\mathfrak{h_i'}^{+}\oplus \mathfrak{h_i'}^{-}$ by Proposition 2 of \cite{Su}, it follows that $$g\mathfrak{h_i'}g^{-1}=g\mathfrak{h_i'}^{+}g^{-1}\oplus g\mathfrak{h_i'}^{-}g^{-1}=g\mathfrak{h_i'}^{+}g^{-1}\oplus k\mathfrak{h_i'}^{-}k^{-1}.$$
  Moreover, it is straightforward to verify that the component $k\mathfrak{h_i'}^{-}k^{-1}$ is precisely the vector part of $k\mathfrak{h_i'}k^{-1}.$  Using the assumption, $\mathfrak{h_i'}=g\mathfrak{h_i'}g^{-1}$, we conclude that $g\cdot \mathfrak{h_i'}^{-}=\mathfrak{h_i'}^{-}$, that is, $k\cdot \mathfrak{h_i'}^{-}=\mathfrak{h_i'}^{-}$.

Note that every element of $Z_{\mathfrak{k}}(\mathfrak{m})$ fixes $\mathfrak{h_i'}^{-}$ point-wise, as $\mathfrak{h_i'}^- \subset \mathfrak{m}$. Consequently, we have $k Z_{\mathfrak{k}}(\mathfrak{m})\cdot\mathfrak{h_i'}^-=\mathfrak{h_i'}^-$. Since $\mathfrak{h_i'}^-= \mathfrak{l_i}^{\perp}$ in $\mathfrak{m}$, it follows that $k Z_{\mathfrak{k}}(\mathfrak{m})\cdot\mathfrak{l_i}=\mathfrak{l_i}$. Therefore, $kZ_{\mathfrak{k}}(\mathfrak{m})$ corresponds to a Weyl group element $s \in W(\mathbf{R})$ that fixes the admissible root systems  associated to the subspace $\mathfrak{l_i}$, say $F_1,F_2,\ldots,F_m$. This proves $(2)$.
\end{proof} 
  
In the following subsections, we examine the real semisimple Lie algebras of types $A_n, B_n, C_n, D_n$ with respect to a standard tuple $(\mathfrak{k},\mathfrak{p},\mathfrak{m})$. Using a suitable Weyl group element associated to the root system of $\mathfrak{m}^{\mathbb{C}}$, that preserves the admissible root systems, we determine an element of $K$ that fixes all the Cartan subalgebras associated to the admissible root systems.

    \begin{subsection}{Real Lie algebra $\mathfrak{g}=\mathfrak{sl}(n,\mathbb{R})$ [$A_n$-type]} \label{sl(n,R)}
The Lie algebra $\mathfrak{g}=\mathfrak{sl}(n,\mathbb{R})= \{X \in \mathfrak{gl}(n,\mathbb{R}): {\rm Tr}(X)=0)\}$. Consider $\mathfrak{k}=\{X\in \mathfrak{g}: X^{t}=-X\}$ and $\mathfrak{p}=\{X\in \mathfrak{g}: X^{t}=X\}$. Then, $\mathfrak{g}=\mathfrak{k}+\mathfrak{p}$ is a Cartan decomposition. Let $\mathfrak{m}=\{D(h_1, h_2, \ldots, h_n): \sum_{i=1}^n h_i=0 \}$ denote the subalgebra of diagonal traceless matrices. Then $\mathfrak{m}$ is a Cartan subalgebra of $\mathfrak{sl}(n,\mathbb{R})$ and $\mathfrak{m}\subset{\mathfrak{p}}$. Furthermore, since $\mathfrak{m}$ consists entirely of symmetric matrices, its ``vector part" coincides with itself; $\mathfrak{m}^-=\mathfrak{m}$.
Define linear functional (character) $e_i$ on $\mathfrak{m}^{\mathbb{C}}$ by $$e_i(D(h_1, h_2, \ldots, h_n))=h_i,\;\;\; \forall \ 1\leq i \leq n.$$ Then the root system associated to $\mathfrak{m}$ is $$\mathbf{R}(\mathfrak{m})=\{\pm(e_i\pm e_j): 1\leq i < j \leq n\}.$$ The Weyl group of $\mathfrak{g}^{\mathbb{C}}$  with respect to $\mathfrak{m}^{\mathbb{C}}$ is isomorphic to the symmetric group $S_n$,  acting by permuting the characters $e_i$. 
       Following Equation (70) of \cite{Su}, consider the following admissible root systems \begin{center}
           $F_0=\phi$\\
           $F_1=\{e_1-e_2\}$\\
           $F_2=\{e_1-e_2, e_3-e_4\}$\\
           and
           $F_k=\{e_1-e_2, e_3-e_4,\ldots, e_{2k-1}-e_{2k}\}$,    
           \end{center} where $k=0,1,\ldots, [\frac{n+1}{2}]$. The Cartan subalgebra $\mathfrak{h_i} $ associated to the admissible root system $F_i$ has the following form. Each Cartan subalgebra $\mathfrak{h_i}$ consists of block diagonal matrices where the first $i$ blocks are $2\times2$ matrices of the form $\begin{pmatrix}
               h_j & \theta_j\\
               -\theta_j & h_j
           \end{pmatrix}$ for $1\leq j \leq i$ and the remaining blocks are of size one. The only permutation that fixes the admissible root system $F_i$, for all $i$, are transpositions of the form $(i,i+1)$, where $i$ is an odd integer,  along with compositions of such transpositions. The transposition $(i,i+1)$ correspond to a Wely group element $\sigma_i$ that sends $e_i$ to $e_{i+1}$, $e_{i+1}$ to $e_{i}$, and fixes all $e_j$ for $j \neq i, i+1$. Consider the block diagonal matrix $J_i$ whose  block$\begin{pmatrix}
               ii & i(i+1)\\
               (i+1)i & (i+1)(i+1)
           \end{pmatrix}=\begin{pmatrix}
               0 & 1\\
               -1 & 0
           \end{pmatrix}$ and 1 otherwise.
           Then $J_i \in K$ and it represents the Weyl group element $\sigma_i$. Therefore, for any Weyl group element that fixes the admissible root systems, one can construct a corresponding element in  $K$ by taking products of some of the $J_i$'s. It is straightforward to verify that each $J_i$ fixes all Cartan subalgebras, and their products do as well.

           \begin{proposition}
            Let $\mathfrak{h_i}$ be the Cartan subalgebra associated to the admissible root system $F_i$ for each $i$ of the Lie algebra $\mathfrak{sl}(n,\mathbb{R})$. Then, there exist non-trivial elements (determined as in Subsection \ref{sl(n,R)}) in $ K$ that fix $\mathfrak{h_i}$ for every $i$.
         \end{proposition}
       
\end{subsection}

   \begin{subsection}
       {Real symplectic Lie algebra $\mathfrak{g}=\mathfrak{sp}(n,\mathbb{R})$ [$C_n$-type]}

       The Lie algebra $\mathfrak{g}=\mathfrak{sp}(n,\mathbb{R})= \{X \in \mathfrak{gl}(2n,\mathbb{R}: X^{t}J+JX=0)\},$ where $X^{t}$ denotes transpose of the matrix $X$ and $J=\begin{pmatrix}
           0_n & I_n\\
           -I_n & 0_n
       \end{pmatrix}$. Here, $0_n \ \text{and} \ I_n$ denote the $n \times n$ zero and identity matrices, respectively. Consider $\mathfrak{k}=\{X\in \mathfrak{g}: X^{t}=-X\}$ and $\mathfrak{p}=\{X\in \mathfrak{g}: X^{t}=X\}$. Then, $\mathfrak{g}=\mathfrak{k}+\mathfrak{p}$ is a Cartan decomposition. The collection of diagonal matrices $\mathfrak{m}=\{D(h_1, h_2, \ldots, h_n,-h_1,-h_2,\ldots, -h_n ): h_i \in \mathbb{R}\}$ is a Cartan subalgebra of $\mathfrak{sp}(n,\mathbb{R})$ and $\mathfrak{m}\subset{\mathfrak{p}}$ that implies $\mathfrak{m}^-=\mathfrak{m}$.
Let $e_i$ denote the character on $\mathfrak{m}^{\mathbb{C}}$ defined as $$e_i(D(h_1, h_2, \ldots, h_n,-h_1,-h_2,\ldots, -h_n ))=h_i,\;\;\; \forall \ 1\leq i \leq n.$$ Then the root system associated to $\mathfrak{m}$ is $$\mathbf{R}(\mathfrak{m})=\{\pm(e_i\pm e_j), (1 \leq i < j \leq n), \pm 2e_i, (1\leq i\leq n)\}.$$ The Weyl group of $\mathfrak{g}^{\mathbb{C}}$ corresponding to $\mathfrak{m}^{\mathbb{C}}$ is isomorphic to $S_n \ltimes \mathbb{Z}_2^n$, where $S_n$ is the symmetric group on $n$ elements and $\mathbb{Z}_2^n$ represents the group of sign changes. The elements of this Weyl group permute the characters $e_i$ and change their signs.
Consider the admissible root system [\cite{Su}, Equation (80)] \begin{center}
           $F(k,l)=(2e_1,2e_2,\ldots,2e_k, e_{k+1}-e_{k+2}, \ldots,e_{k+2l-1}-e_{k+2l} ),$
       \end{center} where $k+2l\leq n$, $k\geq 0$ and $l\geq 0$.
       
\vspace{0.2 cm}
       \noindent\textbf{Claim:} The Weyl group element that fixes all the admissible root systems [see Remark \ref{preserverAdmissibleRootSystem}] other than identity is the reflection that maps each $e_i$ to $-e_i$.

        Let $\sigma$ be the transposition that sends $i$ to $j$ where $i \leq j$ and $i \neq n$. Then, $\sigma$ does not preserve $F(i,0)$. In the case when $i=n$, look for the admissible root systems of type $F(0,l)$(choose a suitable $l$ depending on the image of $i=n$). Therefore, observe that corresponding to every permutation, there exists an admissible system that is not preserved by that permutation. Next assume that $\tau$ be a reflection that does not send each $e_i$ to $-e_i$, in particular, assume that $\tau(e_i)=-e_i$ for all $i\neq 1$. Then look for the admissible root system $F(0,1)=\{e_1-e_2\}$ that is not preserved by $\tau$. Therefore, the only element of the Weyl group that fixes all the admissible root systems is a reflection $\tau$, satisfying $\tau(e_i)=-e_i$ for all $1\leq i\leq n$.

        Observe that the matrix $J\in K$. Also, as $\sigma$ send $e_i$ to $-e_i$ , the matrix $J$ is the corresponding element in $K$ such that $Jl_i=-l_i$ for all $l_i \in \mathfrak{l_i}$. Therefore, $J\mathfrak{l_i}=\mathfrak{l_i}$. Furthermore, $Jh_i\in \mathfrak{h_i}$ for all $h_i \in \mathfrak{h_i}$, where $\mathfrak{h_i}$ is the Cartan subalgebra associated to $F_i$ for all $i$.  Hence, $J$ preserves all the Cartan subalgebras. This proves the following result.

        \begin{proposition}
            Let $\mathfrak{h_i}$ be the Cartan subalgebra associated to the admissible root system $F_i$ for each $i$ of the Lie algebra $\mathfrak{sp}(n,\mathbb{R})$. Then, there exists a unique non-trivial element $J \in K$ that fixes $\mathfrak{h_i}$ for every $i$.
         \end{proposition}
   \end{subsection}

   \begin{subsection}{Real Orthogonal Lie algebra $\mathfrak{g}=\mathfrak{0}(2n+1,\mathbb{R})$ [$B_n$-type]}

The Lie algebra $\mathfrak{g}=\mathfrak{o}(n,\mathbb{R})= \{X \in \mathfrak{gl}(2n,\mathbb{R}: X^{t}J+JX=0)\},$ where $J=\begin{pmatrix}
           0_n & I_n & 0\\
           I_n & 0_n & 0\\
           0 & 0 & -1
       \end{pmatrix}.$ 
       Let $\mathfrak{g}=\mathfrak{k}+\mathfrak{p}$ be a Cartan decomposition as in the previous subsection.
       The collection of diagonal matrices $\mathfrak{m}=\{D(h_1, h_2, \ldots, h_n,-h_1, -h_2, \ldots, -h_n,0) :  h_i \in \mathbb{R}\}$ is a Cartan subalgebra of $\mathfrak{0}(2n+1,\mathbb{R})$ and $\mathfrak{m}\subset{\mathfrak{p}}$ that implies vector part of $\mathfrak{m}$ is equal to $\mathfrak{m}$, that is, $\mathfrak{m}^-=\mathfrak{m}$.  Let $e_i$ denote the character on $\mathfrak{m}^{\mathbb{C}}$ defined as $$e_i(D(h_1, h_2, \ldots, h_n,-h_1,-h_2,\ldots, -h_n,0 ))=h_i,\;\;\; \forall \
        1\leq i \leq n.$$ Then the root system associated to $\mathfrak{m}$ is $$\mathbf{R}(\mathfrak{m})=\{\pm(e_i\pm e_j), 1\leq i < j \leq n; \pm e_i , 1 \leq i \leq n\}.$$ The Weyl group of $\mathfrak{g}^{\mathbb{C}}$ corresponding to $\mathfrak{m}^{\mathbb{C}}$ consists of all permutations of $e_i$'s and change of signs of $e_i$'s. This group acts on the root system by permuting the roots and reflecting them across hyperplanes.
Consider the admissible root systems $F(k,l)$ and $F(k,l)'$ defined as follows [\cite{Su}, Equation (79) and Equation (79)']
       \begin{enumerate}
           \item $F(k,l)=\{e_1+e_2,e_1-e_2,\ldots,e_{2k-1}+e_{2k}, e_{2k-1}-e_{2k},e_{2k+1}-e_{2k+l+1}, \ldots,e_{2k+l}-e_{2k+2l}\},$
        where $2k+2l\leq n$, $k\geq 0$ and $l\geq 0$.

        \item $F(k,l)'=\{e_1+e_2,e_1-e_2,\ldots,e_{2k-1}+e_{2k}, e_{2k-1}-e_{2k},e_{2k+1}-e_{2k+l+1}, \ldots,e_{2k+l}-e_{2k+2l}, e_{2k+2l+1}\},$
        where $2k+2l+1\leq n$, $k\geq 0$ and $l\geq 0$.
       \end{enumerate}

       Similarly to the previous section, one can observe that when $n$ is even, the only element of the Weyl group that fixes all admissible root systems is the element that changes the signs of all characters $e_1, e_2, \ldots, e_n$. Furthermore, when $n$ is odd, there are three elements of Weyl group that fix all the admissible root systems. The elements are $(i)$ that change the sign of character $e_n$ only, $(ii)$ that changes the signs of characters $e_1, e_2, \ldots e_{n-1}$, and $(iii)$ that changes the sign of character $e_i$' s for $1 \leq i \leq n$. These observations highlight the dependence of the Weyl group invariance on the structure of the admissible root systems, particularly the presence of the additional root $e_{2k+2l+1}$ of $F(k,l)'$.  From this observation, we get the following outcome.

       \begin{proposition}
            Let $\mathfrak{h_i}$ be the Cartan subalgebra associated to the admissible root system $F_i$ for each $i$ of the Lie algebra $\mathfrak{o}(2n+1,\mathbb{R})$.  When $n$ is odd, there are two non-trivial elements of $K$ that fixes $\mathfrak{h_i}$ for every $i$ other than the non-trivial element $J \in K$. When $n$ is even, there is an element in $K$ that fixes all the Cartan subalgebras.
         \end{proposition}
    
\end{subsection}

\begin{subsection}{Real Orthogonal Lie algebra $\mathfrak{g}=\mathfrak{0}(2n,\mathbb{R})$ [$D_n$-type]}

 The Lie algebra $\mathfrak{g}=\mathfrak{o}(n,\mathbb{R})= \{X \in \mathfrak{gl}(2n,\mathbb{R}: X^{t}J+JX=0)\},$ where $J=\begin{pmatrix}
           0_n & I_n\\
           I_n & 0_n
       \end{pmatrix}$.   Let $\mathfrak{g}=\mathfrak{k}+\mathfrak{p}$ be a Cartan decomposition as in the previous subsection.
       The collection of diagonal matrices $\mathfrak{m}=\{D(h_1, h_2, \ldots, h_n,-h_1, -h_2, \ldots, -h_n) :  h_i \in \mathbb{R}\}$ is a Cartan subalgebra of $\mathfrak{0}(2n,\mathbb{R})$ and $\mathfrak{m}\subset{\mathfrak{p}}$, that is, $\mathfrak{m}^-=\mathfrak{m}$.
Let $e_i$ be the character on $\mathfrak{m}^{\mathbb{C}}$ defined as $$e_i(D(h_1, h_2, \ldots, h_n,-h_1,-h_2,\ldots, -h_n ))=h_i,\;\;\; \forall \ 1\leq i \leq n.$$ Then the root system associated to $\mathfrak{m}$ is $$\mathbf{R}(\mathfrak{m})=\{\pm(e_i\pm e_j): 1\leq i < j \leq n\}.$$ The Weyl group of $\mathfrak{g}^{\mathbb{C}}$ corresponding to $\mathfrak{m}^{\mathbb{C}}$ consists of all permutations of $e_i$'s and change of signs of an even number of $e_i$'s.
Consider the admissible root system [\cite{Su}, Equation (74) and Equation (75)] 
           $$F(k,l)=\{(e_1+e_2,e_1-e_2,\ldots,e_{2k-1}+e_{2k}, e_{2k-1}-e_{2k},e_{2k+1}-e_{2k+l+1}, \ldots,e_{2k+l}-e_{2k+2l} )\},$$
        where $2k+2l\leq n$, $k\geq 0$ and $l\geq 0$. If $n$ is even, there exists another type of admissible root system $$F=\{F(0,n-2), e_{n-1}+e_n)\}.$$  

       Now, we do some case-wise analysis. In the case when $n=4$, the number of non-conjugate Cartan subalgebra classes is seven. The following is one of the admissible root systems up to conjugation under the Weyl group $W(R)$. \begin{center}
           $F_1:=F(0,0)=\phi, \ F_2:=F(1,0)=\{e_1+e_2, e_1-e_2\}, \ F_3:=F(2,0)=\{e_1+e_2, e_1-e_2, e_3+e_4, e_3-e_4\},\ F_4:=F(0,1)=\{e_1-e_2\},   \ F_5:=F(0,2)=\{e_1-e_3, e_2-e_4\},  \ F_6:=F(1,1)=\{e_1+e_2, e_1-e_2, e_3-e_4\},  \ F=\{e_1-e_3, e_2-e_4, e_3+e_4\}$
       
       \end{center}
       An easy observation tells us that the Weyl group elements that fix all admissible root systems are the permutation $(12)(34)$ and the element that changes the signs of $e_1, e_2, e_3, e_4$. 

       In the case when $n=5$, the number of non-conjugate Cartan subalgebra classes is six. Observe that $F_1, F_2, F_3, F_4, F_5, F_6$ are the admissible root systems and the Weyl group elements that fix $F_i$, $1 \leq i \leq 6$, are the same as in the case when $n=4$.

       Now we move to the case when $n=6$ (respectively, $n=7$). The number of non-conjugate Cartan subalgebra classes is eleven (respectively, 10). The previous admissible root systems $F_i$ for $1\leq i \leq 6$ also appear in this case. In addition, we have the following admissible root systems $F_i$ for $7\leq i \leq 10$  and $F$. 
       \begin{center}
           $ \ F_7:=F(3,0)=\{e_1+e_2, e_1-e_2, e_3+e_4, e_3-e_4, e_5+e_6, e_5-e_6\}, \ F_8:=F(0,3)=\{e_1-e_4, e_2-e_5, e_3-e_6\},  \ F_{9}:=F(2,1)=\{e_1+e_2, e_1-e_2, e_3+e_4,e_3-e_4, e_5-e_6\}, F_{10}:=F(1,2)=\{e_1+e_2, e_1-e_2, e_3-e_5,e_4-e_6\},  \ F=\{e_1-e_5, e_2-e_6, e_5+e_6\}$,
       
       \end{center}
         where $F$ is not part of the admissible root systems in the case of $n=7$. Again, a simple observation tells us that the only Weyl group that fixes all the admissible root systems for $n=6$ (respectively, $n=7$) is the element that changes signs of $e_1, e_2, e_3, e_4, e_5,e_6$ (respectively, the same element).
         A similar observation tells that for the cases when $n\geq 8$, there exists only one Weyl group that fixes all the admissible root systems. When $n$ is even, the element is which changes the sign of $e_i$'s for all $1\leq i \leq n$, and when $n$ is odd, the element is which changes the sign of $e_i$'s for all $1\leq i \leq n-1$. The following outcomes follow from here.

         \begin{proposition}
            Let $\mathfrak{h_i}$ be the Cartan subalgebra associated to the admissible root system $F_i$ for each $i$ of the Lie algebra $\mathfrak{o}(2n,\mathbb{R})$. Then, for $n\geq 6,$ there exists a unique non-trivial element $J \in K$ that fixes $\mathfrak{h_i}$ for every $i$.
         \end{proposition}
            
\end{subsection}

\subsection{Lie algebra $\g_2$}

Consider the Lie algebra $\g_2$ over the field $\mathbb{R}$. There exists a Cartan subalgebra $\mathfrak{h}$ in which the toroidal part $\mathfrak{h}^+$ is zero and the vector part is itself. Therefore, there exists $\mathfrak{m} \subset \mathfrak{p}$ (in Cartan decomposition) such that $\mathfrak{h}=\mathfrak{m}$. With respect to the Cartan subalgebra $\mathfrak{m}$, the root space $\mathbf{R}(\mathfrak{m})=\{\alpha, \beta, \beta-\alpha, 2\beta-3\alpha, \beta-2\alpha, \beta-3\alpha,-\alpha, -\beta, -(\beta-\alpha), -(2\beta-3\alpha), -(\beta-2\alpha), -(\beta-3\alpha)\}$, where $\alpha$ is a short root and $\beta$ is a long root. Consider the following admissible root systems:\begin{center}
    $F_1=\{\phi\}$, \ $F_2=\{\alpha\}$, \ $F_3=\{\beta\}$, \ $F_4=\{\alpha,\beta\}$
\end{center}   (see page 429 type G1 of \cite{Su}). The subgroup $H_1(\mathbf{R})$ of $W(\mathbf{R})$ that fixes $\{F_2, -F_2\}$ is generated by the reflections $\sigma_{\alpha}, \sigma_{2\beta-3\alpha}$ (clear from the diagram given on page 429 of \cite{Su}). Similarly, the subgroup $H_2(\mathbf{R})$ of $W(\mathbf{R})$ that fixes $\{F_3, -F_3\}$ is generated by the reflections $\sigma_{\beta}, \sigma_{\beta-2\alpha}$. It is easy to see that $H_1(\mathbf{R})\cap H_2(\mathbf{R})=\{e,r^3\}$, where $e$ is the identity element and $r$ is the rotation by $60$ degree of $W(\mathbf{R})$. Then, the following result follows from Theorem \ref{s.s liealgebra}.
\begin{proposition}
    The subgroup $H$ of the Weyl group (isomorphic to the dihedral group $D_6$) that fixes all Cartan subalgebras of the Lie algebra $\g_2$ with respect to $F_1, F_2, F_3, F_4$, is the center of the Weyl group.
\end{proposition}

\begin{proof}
    The proof follows from the above description and the following observation. We know that the stabilizer of $\{F_1,-F_1\}$ is the Weyl group $W(\mathbf{R})$. A simple calculation shows that the center $\{e,r^3\}$ of $W(\mathbf{R})$ fixes the set $\{F_4,-F_4\}$ (that is, fixes the admissible root system $f_4$). This proves the theorem.
\end{proof}
\begin{remark}
 It may be an independent interest to find out all the elements of $K$ that fix all Cartan subalgebras $\mathfrak{h_1},\cdots,\mathfrak{h_n}$, taking a representative of each distinct conjugacy class of Cartan subalgebras for all other exceptional types of Lie algebras.
\end{remark}

\section{On characteristic subgroups of Lie groups}

In this section, we prove Theorem \ref{intersection with sub}. We begin with the following lemma.

\begin{lemma}\label{maximal nilpotent}
    Let $G$ be a connected solvable Lie group with nilradical $N$. Let $A$ be a maximal connected nilpotent subgroup of $G$ such that $G=AN$. Then $A$ is a Cartan subgroup of $G$.
\end{lemma}
\begin{proof}
    By Lemma 3.1 of \cite{MS21}, $N_G(A)$ is a connected nilpotent subgroup of $G$. Since $A$ is maximal, we have $N_G(A)=A$. By Corollary 3.2 of \cite{MS21} (or Remark 3.3 of \cite{MS21}), we conclude that $A$ is a Cartan subgroup of $G$.
\end{proof}

\begin{proof} [Proof of Theorem \ref{intersection with sub}:]
Let $R$ be the radical and $N$ the nilradical of $G$. For a given $\Gamma$-stable Cartan subgroup $H$ of $G$, we construct a $\Gamma$-stable Cartan subgroup $H_M$ of $M$ such that $H\cap M\subset H_M$. To enhance the readability of the proof, we divide the proof into the following cases.\\

\noindent{\textbf{Case 1:}} Suppose $M$ is semisimple. Then $G=MZ_G(M)^0$. Any Cartan subgroup of $G$ is of the form $H:=H_MH_{Z_G(M)^0}.$ Therefore, $H\cap M=H_M$, and it follows that $H\cap M$ is $\Gamma$-stable.\\ 

\noindent{\textbf{Case 2:}} Suppose $M$ is the radical $R$ of $G$. Then we have $R=(H\cap R)N$ by Corollary 4.1(2) of \cite{MS21}. Since $H$ is $\Gamma$-stable, $H\cap R$ is too. In view of   \cite[Proposition 3.1]{MS21}, we construct a Cartan subgroup of $R$ containing $H\cap R$. Set $L_1=N_R(H\cap R)$. By \cite[Proposition 3.1]{MS21}, $L_1$ is connected and clearly $\Gamma$-stable. Consequently, we obtain a sequence of connected $\Gamma$-stable subgroups $\{L_{i+1}\}$ of $G$ such that $$H\cap R\subset L_1\subset L_2\subset\cdots,$$ where $L_{i+1}:=N_R(L_i)$. Note that this series terminates after finitely many steps, as $G$ has finite dimension. Proposition 3.1 of \cite{MS21} assures that the final subgroup in this sequence is a Cartan subgroup of $R$, which we denote by $H_R$. Clearly, $H_R$ is $\Gamma$-stable. This proves the result if $M=R$. \\

\noindent{\textbf{Case 3:}} Suppose $M$ is contained in the radical $R$. Then $M=(H\cap M)R_N$, where $R_N$ is the nilradical of $M$. Note that $H\cap R$ and $R_N$ are $\Gamma$-stable. Proceeding exactly as in Case 2 then yields the desired result.\\

\noindent{\textbf{Case 4:}} Suppose $M$ contains the radical $R$. Let $S$ be a Levi subgroup of $G$ such that $H=(H\cap S)(H\cap R)$. Set $J:=(S\cap M)^0$. Then $M=JR$, and $J$ is a connected normal subgroup of $S$. By \cite{I}, we have $S=JZ_S(J)^0$ (an almost direct product). Accordingly, it is easy to see that $H_S:=H\cap S=H_JH_{Z_S(J)^0}$, where  $H_J$ and $H_{Z_S(J)^0}$ are Cartan subgroups of $J$ and $Z_S(J)^0,$ respectively.  
    Since $H$ is $\Gamma$-stable, it follows that $H_J(H\cap R)$, and $H_J^0(H\cap R)$ is also $\Gamma$-stable. 
    Note that $H_J^0R$ is a closed connected solvable Lie group and $H_J^0R=H_J^0(H\cap R)N$. Since $H_J^0R$ and $H_J^0(H\cap R)$ are $\Gamma$-stable, there exits a $\Gamma$-stable Cartan subgroup $H'$ of $H_J^0R$ such that $H_J^0(H\cap R)\subset H'$ by Case 2. Note that $H'$ is a maximal connected nilpotent subgroup of $H_J^0R$. Now, we set $H_M:=H'Z_M(H')$. Clearly $H_M$ is $\Gamma$-stable.
    
    \vspace{0.2 cm}
    
\noindent
{\it Claim:} $H_M$ is a Cartan subgroup of $M$.
\vspace{0.2 cm}
    
\noindent
To prove our claim, we show that $H_M=H_JH_{Z_R(H_J)}$ for some Cartan subgroup $H_{Z_R(H_J)}$ of the closed connected subgroup $Z_R(H_J)$ of $R$ (Proposition 3.6 of \cite{MS21}). Note that, $H'=H_J^0(H'\cap R)$.
\vspace{0.2 cm}

\noindent{\it Step 1:} Note that $H'$ is a maximal connected nilpotent subgroup of $M$. Indeed, since $H'$ is a Cartan subgroup of $H_J^0R$, it follows that $H'$ is a Cartan subgroup of $H'R$ by Lemma \ref{connected component}. Now, suppose $H''$ is a maximal connected nilpotent subgroup of $M$ containing $H'$, then we have $H_J^0R/R=H'R/R\subset H''R/R$. Since $H_J^0R/R$ is a maximal connected nilpotent in $M/R$, it implies that $H'R/R=H''R/R.$ Therefore,
we have $H'\subset H''\subset H'R$. As $H'$ is a Cartan subgroup of $H'R$, it is a maximal nilpotent group, implying that $H'=H''$.
\vspace{0.2 cm}

\noindent
{\it Step 2:} Now we prove $H'\cap R\subset Z_R(H_J^0).$
Note that $[H',H']\subset H'\cap N$, and $H_J^0$ acts trivially on $(H'\cap R)/(H'\cap N)$. Since ${\rm Ad}(H_J^0)$ consisting of semisimple elements, we get that $H'=H_J^0Z_{H'\cap R}(H_J^0)(H'\cap N)$.
Therefore, to prove our assertion, it is enough to show that $H'\cap N\subset Z_N(H_J^0)$, in other words, the action of $H_J^0$ on $H'\cap N$ is trivial. Since $H'\cap R\subset Z_R(H_J^0)$, and $\overline{[H',H']}\subset (H'\cap N).$
Let $$\pi:(H'\cap N)\to (H'\cap N)/\overline{[H',H']}$$  be the projection map. Note that $\overline{[H', H']}$ is a proper closed connected subgroup of the connected group $H'\cap N$. Indeed, if $\overline{[H',H']}=(H'\cap N)$, then it follows that $\overline{[H',[H',H']]}$ is equal to $H'\cap N$, contradicting that $\overline{H'}$ is a connected solvable subgroup. Clearly, $H_J^0$ commutes with $H'\cap N$ modulo $\overline{[H',H']}.$ Therefore, we have
$$\pi^{-1}(Z_{(H'\cap N)/[H', H']}(H_J^0))=Z_{H'\cap N}(H_J^0)\overline{[H',H']},$$ by Proposition 3.4 of \cite{MS21}. This implies that $H'\cap N=Z_{(H'\cap N)}(H_J^0)\overline{[H', H']}.$ As $\overline{[H', H']}\subset (H'\cap N)$ is strictly of lower dimension by repeating the argument, we conclude that $H'\cap N=Z_{H'\cap N}(H_J^0).$ 
Therefore, $$H'\cap R=Z_{H'\cap R}(H_J^0)(H'\cap N)\subset Z_R(H_J^0)Z_N(H_J^0)=Z_R(H_J^0),$$ as required.
\vspace{0.2 cm}

\noindent
{\it Step 3:} In this step, we show that $H'\cap R$ is a Cartan subgroup of $Z_R(H_J^0)=Z_R(H_J)$ (see Lemma 3.9 of \cite{MS21}). Let $\tilde H$ be a maximal nilpotent subgroup of $Z_R(H_J^0)$ containing  $H'\cap R.$ Then $H_J^0\tilde H$ is a maximal nilpotent subgroup of $H_J^0R$. Since $H'$ is a Cartan subgroup of $H_J^0R$, we have $H'=H_J^0\tilde H$. This proves that $H'\cap R$ is a maximal nilpotent subgroup of $Z_R(H_J^0).$ Since $Z_R(H_J^0)=(H'\cap R)(N\cap Z_R(H_J^0))$ and $H'\cap R$ is maximal nilpotent, by Lemma \ref{maximal nilpotent}, it follows that $H'\cap R$ is a Cartan subgroup of $Z_R(H_J^0)$.
\vspace{0.2 cm}

\noindent
{\it Step 4:} By Step 3, there exists a Cartan subgroup $H_{Z_R(H_J)}$ of $Z_R(H_J)$ such that $H'\cap R=H_{Z_R(H_J)}$. Hence $H'=H_J^0(H'\cap R)=H_J^0H_{Z_R(H_J^0)}=H_J^0H_{Z_R(H_J)},$ and that $H_JH_{Z_R(H_J)}$ forms a Cartan subgroup of $M$. Next, we show $$H_JH_{Z_R(H_J)}=Z_M(H')H'.$$ Since $Z_J(H_J^0)=H_J$,  we have $Z_M(H')H'=Z_{JR}(H')H'\subset H_JH'=H_J(H'\cap R)=H_JH_{Z_R(H_J)}.$ As $H'\cap R\subset Z_R(H_J)$ and $Z_J(H_J^0)=H_J$ as $J$ is semisimple, it follows that $H_J\subset Z_M(H')$, and hence $H_JH_{Z_R(H_J)}\subset Z_M(H')H'$. Combining both inclusions, we obtain the desired equality. 
\vspace{0.2 cm}

\noindent
 Finally, note that $H\cap M=H_J(H\cap R)\subset H_M$. This proves the result in Case 4.
\vspace{0.2 cm}

 \noindent{\textbf{Case 5:}} Suppose $M$ is neither semisimple nor solvable. Then $M$ has a non-trivial radical, say $R_M$. Since $R_M$ is a characteristic subgroup of $M$, and $M$ is normal in $G$, we have $R_M\subset R$, and $R_M$ is $\Gamma$-stable. Consequently, $H\cap R_M$ is $\Gamma$-stable. Let $S$ be a Levi subgroup of $G$ such that $H=(H\cap S)(H\cap R)$, and set $J:=(M\cap S)^0$. Then $M=JR_M$. Let $H_J$ be a Cartan subgroup of $J$. Then $H\cap S=H_JH_{Z_J(S)^0}$. Note that both $H_J^0R_M$ and $H_JM$ are closed subgroups of $M$.
 Now we consider the $\Gamma$-stable Lie subgroup $H_J^0R_M$. Clearly, $H\cap M\subset H_J(H\cap R_M)$, and $(H\cap R_M)$ is a $\Gamma$-stable and connected. Since $H_J^0(H\cap R_M)$ is $\Gamma$-stable, and $H_J^0R_M=H_J^0(H\cap R_M)R_N$, where $R_N$ is the nilradical of $M$, there exists a $\Gamma$-stable Cartan subgroup $H'$ of $H_J^0R_M$ containing 
   $H_J^0(H\cap R_M)$. Arguing as in the proof of Case 4 in the group $M=JR_M$, we conclude that $$H_M:=Z_M(H')H'=H_JH_{Z_{R_M}(H_J)},$$ is a $\Gamma$-stable Cartan subgroup of $M$, which contains $H\cap M=H_J(H\cap R_M).$ This proves the result.

\end{proof}

Now we provide an example of a $\Gamma$-stable Cartan subgroup $H_M$ of $M$ in which $H\cap M$ is properly contained.
\begin{example}
	Let $G=(SL(2,\mathbb R)\times SL(2,\mathbb R))\ltimes (\mathbb R^2\oplus \mathbb{H})$ be a group defined by the action of $SL(2,\mathbb R)\times SL(2,\mathbb R)$ on $\mathbb R^2\oplus \mathbb{H}$, given by $$(A,B)(u,v):=(Au, Bv),\;\forall A, B\in SL(2,\mathbb R),\;\; u\in\mathbb R^2,\;\; v\in \mathbb{H}.$$
    Here, the $3$-dimensional Heisenberg group  \begin{center}
	$\mathbb{H}=\Big\{ \begin{pmatrix} 1 & a & c\\
		0 & 1 & b\\
		0 & 0 & 1
	\end{pmatrix}: a, b, c \in \mathbb{R} \Big\}. $
\end{center} 
The action of $SL(2,\mathbb R)$ on the nilpotent group $\mathbb H$ is canonically defined as a group of automorphisms of $\mathbb H$ such that the center $Z=Z(\mathbb H)$ of $\mathbb H$ is point-wise fixed and the quotient action on $\mathbb H/Z$ ,
which we view as $\mathbb R^2$, is the usual linear action of $SL(2,\mathbb R)$ on $\mathbb R^2.$ More precisely,
the action of $SL(2,\R)$ on $\mathbb{H}$ is given by
\begin{center}
$	\begin{pmatrix}
		x_{11} & x_{12}\\
		x_{21} & x_{22}
	\end{pmatrix} \begin{pmatrix}
	1 & a & c\\
	0 & 1 & b\\
	0 & 0 & 1
\end{pmatrix}= \begin{pmatrix}
1 & x_{11}a+x_{12}b & \frac{1}{2}(x_{11}a+x_{12}b)(x_{21}a+x_{22}b)+c-\frac{1}{2} ab\\
0 & 1 & x_{21}a+x_{22}b\\
0 & 0 & 1
\end{pmatrix} .$
\end{center} 
Let $M=SL(2,\mathbb R)\ltimes (\mathbb R^2\oplus \mathbb{H})$, where the action of $SL(2,\mathbb R)$ on $\mathbb R^2\oplus \mathbb{H}$ is given by $B(u,v):=(u, Bv),\;\;\forall \; B\in SL(2,\mathbb R),\;\; u\in\mathbb R^2,\;\; v\in \mathbb{H} .$ Let $\Gamma$ denote the group of automorphisms of $G$. Note that $M$ is a $\Gamma$-stable (non-solvable, non-semisimple) closed connected normal subgroup of $M$. Let $H_1\times H_2$ be any Cartan subgroups of $SL(2,\mathbb R)\times SL(2,\mathbb R)$. The radical $R$ of $G$ is $R:=\mathbb R^2\oplus \mathbb{H}$, and $Z_R(H_1\times H_2)=(\{0\},Z)$, where $Z:=Z(\mathbb{H})$. Since the center of $\mathbb{H}$ is abelian, $(\{0\},Z)$ is a Cartan subgroup of $Z_R(H_1\times H_2)$. Therefore $$H:=H_1\times H_2\times (\{0\}\oplus Z)$$  is a Cartan subgroup of $G$, and hence $H\cap R:=(\{0\},Z)$. Note that a Levi subgroup of $G$ is $S=SL(2,\mathbb R)\times SL(2,\mathbb R)$ and $H=(H\cap S)(H\cap R)$.
	Let $H_2$ be a Cartan subgroup of $SL(2,\mathbb R)$ ( the second component of the semisimple part of $G$). Note that $Z_R(H_2)=\mathbb R^2\oplus Z$, where $Z=Z(\mathbb H)$. Then $$H_M:=H_2\times (\mathbb R^2\oplus Z)$$ is a Cartan subgroup of $M$. 
	Also, we have $H\cap M=H_2\times (\{0\}\oplus Z).$
	 Clearly,
	$H\cap M\subset H_M,$ and $H_M$ is $\Gamma$-stable, follows from  Theorem \ref{intersection with sub} (see Case 4).
\end{example}

\noindent
{\bf Acknowledgements:}
The second author gratefully acknowledges support from the SRIC grant at IIT Roorkee. The third author’s research is supported by the NBHM Postdoctoral Fellowship (PDF no. 0204/27/(29)/2023/R\&D‑II/11930).
 We would like to thank Riddhi Shah for her encouragement of this project and for many useful discussions and comments, especially in the proof of Theorem \ref{intersection with sub}, which improved the statement from our earlier version.

\medskip

\noindent
{\bf Conflict of interest:}
On behalf of all authors, the corresponding author states that there is no conflict of interest.

    \begin{flushleft}
Parteek Kumar and Arunava Mandal\\
Department of Mathematics\\
Indian Institute of Technology Roorkee\\
Uttarakhand 247667, India\\

\medskip
E-mail: {\tt parteek$\textunderscore$k@ma.iitr.ac.in} and {\tt arunava@ma.iitr.ac.in} 

\medskip

Shashank Vikram Singh\\
Department of Mathematics\\
Indian Institute of Science Education and Research, Mohali \\
Knowledge City, Sector -81, Mohali, Punjab.-140306, India\\

\medskip
E-mail: {\tt shashank@iisermohali.ac.in} 

\end{flushleft}


\begin{thebibliography}{99}

\bibitem{BhM} S. Bhaumik and A. Mandal: On the density of images of the power maps in Lie
groups, Arch. Math. (Basel) 110 (2018), no. 2, 115–130.
	
	\bibitem{BM}  A. Borel and G. D. Mostow: On semi-simple automorphisms of Lie algebras, Ann. of Math. 61 (1955), 389-504.

 \bibitem{BS} A. Borel, J. P. Serre: Sur certain ssous-groupes des groupes de Lie compacts. Comment. Math. Helv. 27 (1953), 128–139. 
	
\bibitem{Ch} P. Chatterjee: Automorphism invariant Cartan subgroups and power
maps of disconnected groups, Math. Z. 269, (2011) 221-233. 

\bibitem{Che55} C. Chevalley: Th\'eorie des Groupes de Lie, III, Hermann, Paris (1955). 
	
\bibitem{D}  S. G. Dani:  Actions of automorphism groups of Lie groups.Adv. Lect. Math. (ALM), 41
International Press, Somerville, MA, (2018) 529–562.



\bibitem{I} Iwasawa, K.: On some types of topological groups. Ann. Math. 50, 507–558 (1949)

\bibitem{K} Anthony W. Knapp: Lie Groups
Beyond an Introduction, Progress in Mathematics
Volume 140.

\bibitem{KM} P. Kumar and A. Mandal: Roots of elements for groups over local fields, arXiv:2503.16987

\bibitem{Ma}  A. Mandal: Dense images of the power maps for a disconnected real algebraic group. J. Group Theory. 24 (2021), 973–985.
	
\bibitem{M} G. D. Mostow: Fully reducible subgroups of algebraic groups. Amer. J. Math. 78 (1956), 200-221.


	 
\bibitem{MS21} A. Mandal and R. Shah: The structure of Cartan
subgroups in Lie groups. Math. Z. 299 (2021), 1587–1606.

\bibitem{MS23} A. Mandal and R. Shah:  Cartan subgroups in connected locally compact groups, Preprint. arXiv:2310.15564 


\bibitem{St} R. Steinberg: Endomorphisms of linear algebraic groups, Memoirs of the Amer.Math.Soc., No. 80, American Mathematical Society, Providence, R. I. (1968).

\bibitem{Su} M. Sugiura: Conjugate classes of Cartan subalgebras in real semisimple Lie algebras, Journal of the Mathematical Society of Japan, Vol. 11, No. 4, 1959.
	

	
\bibitem{W} D. J. Winter: Fixed points and stable subgroups of algebraic group automorphisms. Proc. Amer. Math. Soc. 18 (1967), 1107–1113.
	

\bibitem{Wu} Ta Sun Wu: On the automorphism group of a connected locally compact topological group. Proc. Edinburgh Math. Soc. (2) 35 (1992), no. 2, 285-294.
	
\bibitem{Wus}  M. Wustner: A generalization of the Jordan decomposition,
	Forum Mathematicum 15 (2003),  395--408.
	
	
	\end{thebibliography}
\end{document}